\documentclass[a4paper,10pt]{amsart}

\usepackage[utf8]{inputenc}
\usepackage{url}

\usepackage[expansion=false]{microtype}
\usepackage{amssymb,amsmath,amsopn,amsxtra,amsthm,amsfonts}
\usepackage{xr}
\usepackage[mathcal]{euscript}
\usepackage{mathalfa}

\usepackage[english, status=final, nomargin, inline]{fixme}
\fxusetheme{color}

\usepackage[pdfusetitle,unicode,hidelinks]{hyperref}

\usepackage{enumitem}
\setlist[enumerate,1]{label={(\arabic*)},itemsep=\parskip} 
\setlist[itemize,1]{itemsep=\parskip} 
\newlist{thmlist}{enumerate}{2}
\setlist[thmlist,1]{label={\em(\roman*)},ref={(\roman*)},%
  itemsep=\parskip,leftmargin=*,align=left}
\setlist[thmlist,2]{label={\em(\alph*)},ref={(\alph*)},%
  itemsep=\parskip,leftmargin=*,align=left,topsep=0.1cm}
\newlist{remlist}{enumerate}{2}
\setlist[remlist,1]{label={(\roman*)},ref={(\roman*)},itemsep=\parskip,%
  leftmargin=*,align=left}
\setlist[remlist,2]{label={(\alph*)},ref={(\alph*)},itemsep=\parskip,%
  leftmargin=*,align=left,topsep=0.1cm}

\makeatletter
\let\c@equation\c@subsubsection

\makeatother

\newtheorem{cor}[subsubsection]{Corollary}
\newtheorem{lem}[subsubsection]{Lemma}
\newtheorem{prop}[subsubsection]{Proposition}

\newtheorem{thm}[subsubsection]{Theorem}

\newtheorem*{claim*}{Claim}

\theoremstyle{definition}

\theoremstyle{remark}
\newtheorem{rem}[subsubsection]{Remark}

\newtheorem{exam}[subsubsection]{Example}
\newtheorem{notat}[subsubsection]{Notation}

\newcommand{\thmref}[1]{Theorem~\ref{#1}}

\newcommand{\ssecref}[1]{\sectsign\ref{#1}}

\newcommand{\lemref}[1]{Lemma~\ref{#1}}
\newcommand{\propref}[1]{Proposition~\ref{#1}}
\newcommand{\corref}[1]{Corollary~\ref{#1}}

\newcommand{\remref}[1]{Remark~\ref{#1}}

\newcommand{\notatref}[1]{Notation~\ref{#1}}

\renewcommand{\eqref}[1]{(\ref{#1})}
\newcommand{\itemref}[1]{\ref{#1}}

\usepackage{tikz}
\usetikzlibrary{matrix}
\usepackage{tikz-cd}

\newcommand{\changelocaltocdepth}[1]{%
  \addtocontents{toc}{\protect\setcounter{tocdepth}{#1}}%
  \setcounter{tocdepth}{#1}}

\usepackage{xspace}

\usepackage{xifthen}

\usepackage{xparse}

\newcommand{\nc}{\newcommand}
\nc{\renc}{\renewcommand}
\nc{\ssec}{\subsection}
\nc{\sssec}{\subsubsection}
\nc{\on}{\operatorname}
\nc{\term}[1]{#1\xspace}

\DeclareMathSymbol{A}{\mathalpha}{operators}{`A}
\DeclareMathSymbol{B}{\mathalpha}{operators}{`B}
\DeclareMathSymbol{C}{\mathalpha}{operators}{`C}
\DeclareMathSymbol{D}{\mathalpha}{operators}{`D}
\DeclareMathSymbol{E}{\mathalpha}{operators}{`E}
\DeclareMathSymbol{F}{\mathalpha}{operators}{`F}
\DeclareMathSymbol{G}{\mathalpha}{operators}{`G}
\DeclareMathSymbol{H}{\mathalpha}{operators}{`H}
\DeclareMathSymbol{I}{\mathalpha}{operators}{`I}
\DeclareMathSymbol{J}{\mathalpha}{operators}{`J}
\DeclareMathSymbol{K}{\mathalpha}{operators}{`K}
\DeclareMathSymbol{L}{\mathalpha}{operators}{`L}
\DeclareMathSymbol{M}{\mathalpha}{operators}{`M}
\DeclareMathSymbol{N}{\mathalpha}{operators}{`N}
\DeclareMathSymbol{O}{\mathalpha}{operators}{`O}
\DeclareMathSymbol{P}{\mathalpha}{operators}{`P}
\DeclareMathSymbol{Q}{\mathalpha}{operators}{`Q}
\DeclareMathSymbol{R}{\mathalpha}{operators}{`R}
\DeclareMathSymbol{S}{\mathalpha}{operators}{`S}
\DeclareMathSymbol{T}{\mathalpha}{operators}{`T}
\DeclareMathSymbol{U}{\mathalpha}{operators}{`U}
\DeclareMathSymbol{V}{\mathalpha}{operators}{`V}
\DeclareMathSymbol{W}{\mathalpha}{operators}{`W}
\DeclareMathSymbol{X}{\mathalpha}{operators}{`X}
\DeclareMathSymbol{Y}{\mathalpha}{operators}{`Y}
\DeclareMathSymbol{Z}{\mathalpha}{operators}{`Z}

\nc{\sA}{\ensuremath{\mathcal{A}}\xspace}
\nc{\sB}{\ensuremath{\mathcal{B}}\xspace}
\nc{\sC}{\ensuremath{\mathcal{C}}\xspace}
\nc{\sD}{\ensuremath{\mathcal{D}}\xspace}
\nc{\sE}{\ensuremath{\mathcal{E}}\xspace}
\nc{\sF}{\ensuremath{\mathcal{F}}\xspace}
\nc{\sG}{\ensuremath{\mathcal{G}}\xspace}
\nc{\sH}{\ensuremath{\mathcal{H}}\xspace}
\nc{\sI}{\ensuremath{\mathcal{I}}\xspace}
\nc{\sJ}{\ensuremath{\mathcal{J}}\xspace}
\nc{\sK}{\ensuremath{\mathcal{K}}\xspace}
\nc{\sL}{\ensuremath{\mathcal{L}}\xspace}
\nc{\sM}{\ensuremath{\mathcal{M}}\xspace}
\nc{\sN}{\ensuremath{\mathcal{N}}\xspace}
\nc{\sO}{\ensuremath{\mathcal{O}}\xspace}
\nc{\sP}{\ensuremath{\mathcal{P}}\xspace}
\nc{\sQ}{\ensuremath{\mathcal{Q}}\xspace}
\nc{\sR}{\ensuremath{\mathcal{R}}\xspace}
\nc{\sS}{\ensuremath{\mathcal{S}}\xspace}
\nc{\sT}{\ensuremath{\mathcal{T}}\xspace}
\nc{\sU}{\ensuremath{\mathcal{U}}\xspace}
\nc{\sV}{\ensuremath{\mathcal{V}}\xspace}
\nc{\sW}{\ensuremath{\mathcal{W}}\xspace}
\nc{\sX}{\ensuremath{\mathcal{X}}\xspace}
\nc{\sY}{\ensuremath{\mathcal{Y}}\xspace}
\nc{\sZ}{\ensuremath{\mathcal{Z}}\xspace}

\nc{\bA}{\ensuremath{\mathbf{A}}\xspace}
\nc{\bB}{\ensuremath{\mathbf{B}}\xspace}
\nc{\bC}{\ensuremath{\mathbf{C}}\xspace}
\nc{\bD}{\ensuremath{\mathbf{D}}\xspace}
\nc{\bE}{\ensuremath{\mathbf{E}}\xspace}
\nc{\bF}{\ensuremath{\mathbf{F}}\xspace}
\nc{\bG}{\ensuremath{\mathbf{G}}\xspace}
\nc{\bH}{\ensuremath{\mathbf{H}}\xspace}
\nc{\bI}{\ensuremath{\mathbf{I}}\xspace}
\nc{\bJ}{\ensuremath{\mathbf{J}}\xspace}
\nc{\bK}{\ensuremath{\mathbf{K}}\xspace}
\nc{\bL}{\ensuremath{\mathbf{L}}\xspace}
\nc{\bM}{\ensuremath{\mathbf{M}}\xspace}
\nc{\bN}{\ensuremath{\mathbf{N}}\xspace}
\nc{\bO}{\ensuremath{\mathbf{O}}\xspace}
\nc{\bP}{\ensuremath{\mathbf{P}}\xspace}
\nc{\bQ}{\ensuremath{\mathbf{Q}}\xspace}
\nc{\bR}{\ensuremath{\mathbf{R}}\xspace}
\nc{\bS}{\ensuremath{\mathbf{S}}\xspace}
\nc{\bT}{\ensuremath{\mathbf{T}}\xspace}
\nc{\bU}{\ensuremath{\mathbf{U}}\xspace}
\nc{\bV}{\ensuremath{\mathbf{V}}\xspace}
\nc{\bW}{\ensuremath{\mathbf{W}}\xspace}
\nc{\bX}{\ensuremath{\mathbf{X}}\xspace}
\nc{\bY}{\ensuremath{\mathbf{Y}}\xspace}
\nc{\bZ}{\ensuremath{\mathbf{Z}}\xspace}

\nc{\dA}{\ensuremath{\mathds{A}}\xspace}
\nc{\dB}{\ensuremath{\mathds{B}}\xspace}
\nc{\dC}{\ensuremath{\mathds{C}}\xspace}
\nc{\dD}{\ensuremath{\mathds{D}}\xspace}
\nc{\dE}{\ensuremath{\mathds{E}}\xspace}
\nc{\dF}{\ensuremath{\mathds{F}}\xspace}
\nc{\dG}{\ensuremath{\mathds{G}}\xspace}
\nc{\dH}{\ensuremath{\mathds{H}}\xspace}
\nc{\dI}{\ensuremath{\mathds{I}}\xspace}
\nc{\dJ}{\ensuremath{\mathds{J}}\xspace}
\nc{\dK}{\ensuremath{\mathds{K}}\xspace}
\nc{\dL}{\ensuremath{\mathds{L}}\xspace}
\nc{\dM}{\ensuremath{\mathds{M}}\xspace}
\nc{\dN}{\ensuremath{\mathds{N}}\xspace}
\nc{\dO}{\ensuremath{\mathds{O}}\xspace}
\nc{\dP}{\ensuremath{\mathds{P}}\xspace}
\nc{\dQ}{\ensuremath{\mathds{Q}}\xspace}
\nc{\dR}{\ensuremath{\mathds{R}}\xspace}
\nc{\dS}{\ensuremath{\mathds{S}}\xspace}
\nc{\dT}{\ensuremath{\mathds{T}}\xspace}
\nc{\dU}{\ensuremath{\mathds{U}}\xspace}
\nc{\dV}{\ensuremath{\mathds{V}}\xspace}
\nc{\dW}{\ensuremath{\mathds{W}}\xspace}
\nc{\dX}{\ensuremath{\mathds{X}}\xspace}
\nc{\dY}{\ensuremath{\mathds{Y}}\xspace}
\nc{\dZ}{\ensuremath{\mathds{Z}}\xspace}

\nc{\bbA}{\ensuremath{\mathbb{A}}\xspace}
\nc{\bbB}{\ensuremath{\mathbb{B}}\xspace}
\nc{\bbC}{\ensuremath{\mathbb{C}}\xspace}
\nc{\bbD}{\ensuremath{\mathbb{D}}\xspace}
\nc{\bbE}{\ensuremath{\mathbb{E}}\xspace}
\nc{\bbF}{\ensuremath{\mathbb{F}}\xspace}
\nc{\bbG}{\ensuremath{\mathbb{G}}\xspace}
\nc{\bbH}{\ensuremath{\mathbb{H}}\xspace}
\nc{\bbI}{\ensuremath{\mathbb{I}}\xspace}
\nc{\bbJ}{\ensuremath{\mathbb{J}}\xspace}
\nc{\bbK}{\ensuremath{\mathbb{K}}\xspace}
\nc{\bbL}{\ensuremath{\mathbb{L}}\xspace}
\nc{\bbM}{\ensuremath{\mathbb{M}}\xspace}
\nc{\bbN}{\ensuremath{\mathbb{N}}\xspace}
\nc{\bbO}{\ensuremath{\mathbb{O}}\xspace}
\nc{\bbP}{\ensuremath{\mathbb{P}}\xspace}
\nc{\bbQ}{\ensuremath{\mathbb{Q}}\xspace}
\nc{\bbR}{\ensuremath{\mathbb{R}}\xspace}
\nc{\bbS}{\ensuremath{\mathbb{S}}\xspace}
\nc{\bbT}{\ensuremath{\mathbb{T}}\xspace}
\nc{\bbU}{\ensuremath{\mathbb{U}}\xspace}
\nc{\bbV}{\ensuremath{\mathbb{V}}\xspace}
\nc{\bbW}{\ensuremath{\mathbb{W}}\xspace}
\nc{\bbX}{\ensuremath{\mathbb{X}}\xspace}
\nc{\bbY}{\ensuremath{\mathbb{Y}}\xspace}
\nc{\bbZ}{\ensuremath{\mathbb{Z}}\xspace}

\nc{\mrm}[1]{\ensuremath{\mathrm{#1}}\xspace}
\nc{\mbf}[1]{\ensuremath{\mathbf{#1}}\xspace}
\nc{\mcal}[1]{\ensuremath{\mathcal{#1}}\xspace}
\nc{\msc}[1]{\ensuremath{\mathscr{#1}}\xspace}

\renc{\bar}[1]{\overline{#1}}

\let\sectsign\S
\let\S\relax

\nc{\sub}{\subset}
\nc{\too}{\longrightarrow}
\nc{\hook}{\hookrightarrow}
\nc*{\hooklongrightarrow}{\ensuremath{\lhook\joinrel\relbar\joinrel\rightarrow}}
\nc{\hooklong}{\hooklongrightarrow}
\nc{\twoheadlongrightarrow}{\relbar\joinrel\twoheadrightarrow}
\nc{\shiso}{\approx}
\nc{\isoto}{\xrightarrow{\sim}}
\nc{\isofrom}{\xleftarrow{\sim}}
\renc{\ge}{\geqslant}
\renc{\le}{\leqslant}
\renc{\geq}{\geqslant}
\renc{\leq}{\leqslant}

\nc{\id}{\mathrm{id}}

\DeclareMathOperator{\Hom}{\mathrm{Hom}}
\nc{\uHom}{\underline{\smash{\Hom}}}
\DeclareMathOperator{\Maps}{\mathrm{Maps}}
\DeclareMathOperator{\Aut}{\mathrm{Aut}}
\DeclareMathOperator{\End}{\mathrm{End}}

\nc{\Pre}{\mathrm{PSh}{}}
\nc{\uEnd}{\underline{\smash{\End}}}

\renc{\lim}{\operatorname*{lim}}
\nc{\colim}{\operatorname*{colim}}
\nc{\Cofib}{\on{Cofib}}
\nc{\Fib}{\on{Fib}}
\nc{\initial}{\varnothing}
\nc{\op}{\mathrm{op}}

\DeclareMathOperator*{\fibprod}{\times}

\renc{\coprod}{\sqcup}

\nc{\bDelta}{\mbf{\Delta}}
\nc{\DM}{\mbf{DM}}
\nc{\eff}{\mathrm{eff}}
\nc{\veff}{\mathrm{veff}}
\nc{\cyc}{{\mrm{cyc}}}
\nc{\corr}{{\on{corr}}}
\nc{\ft}{\mrm{ft}}
\nc{\flf}{\mrm{flf}}
\nc{\fet}{{\mrm{f\acute et}}}
\nc{\fsyn}{{\mrm{fsyn}}}
\nc{\syn}{{\mrm{syn}}}
\nc{\lci}{{\mrm{lci}}}
\nc{\Perf}{\mbf{Perf}}
\nc{\perf}{\mrm{perf}}
\nc{\oblv}{\mrm{oblv}}
\nc{\exact}{\on{exact}}

\nc{\F}{{\on{F}}}
\nc{\clopen}{{\mrm{clopen}}}
\nc{\B}{\mrm{B}}
\nc{\D}{\mrm{D}}
\nc{\Fin}{\on{Fin}}
\nc{\Cut}{\on{Cut}}
\nc{\Cart}{\on{Cart}}
\nc{\pairs}{\mathsf{pairs}}
\nc{\Pairs}{\mathrm{Pair}}
\nc{\Trip}{\mathrm{Trip}}
\nc{\Lab}{\mathrm{Lab}}
\nc{\SL}{\mathrm{SL}}
\nc{\coCart}{\mathrm{coCart}}
\nc{\RKE}{\mathrm{RKE}}
\nc{\strict}{\mathrm{strict}}
\nc{\Emb}{\mathrm{Emb}}
\nc{\Split}{\mathrm{Split}}
\nc{\Set}{\mathrm{Set}}
\nc{\sSets}{\mathrm{sSets}}
\nc{\pb}{\mathrm{pb}}
\nc{\fib}{\mathrm{fib}}
\nc{\diff}{\mrm{diff}}
\nc{\gp}{\mrm{gp}}

\nc{\mgp}{\mrm{mot-gp}}
\nc{\FSyn}{\mrm{FSyn}}
\nc{\FEt}{\mrm{FEt}}
\nc{\Spc}{\mrm{Spc}}
\nc{\Ob}{\mrm{Ob}}
\nc{\Spt}{\mrm{Spt}}
\nc{\T}{\bT}
\nc{\suspinf}{\Sigma^\infty}
\nc{\h}{\mrm{h}}
\nc{\uhom}{\underline{\mathrm{Hom}}}
\nc{\umap}{\underline{\mathrm{Maps}}}
\renc{\H}{\bH}
\nc{\Einfty}{{\sE_\infty}}
\nc{\Eone}{{\sE_1}}
\nc{\Stab}{\mrm{Stab}}
\nc{\lax}{{\mrm{lax}}}
\nc{\cocart}{{\mrm{cocart}}}
\nc{\Sch}{\on{Sch}}
\nc{\Fr}{\on{Fr}}
\nc{\A}{\mathbf{A}}
\nc{\N}{\mathbf{N}}
\nc{\Z}{\mathbf{Z}}
\nc{\Q}{\mathbf{Q}}
\nc{\Oo}{\mathcal{O}} 
\nc{\Ll}{\mathcal{L}} 
\nc{\Mm}{\mathcal{M}} 
\nc{\mm}{\mathrm{m}} 
\nc{\K}{\mrm{K}} 
\nc{\W}{\mrm{W}} 
\nc{\red}{{\on{red}}}
\nc{\Voev}{{\on{Voev}}}
\nc{\Corr}{\mrm{Corr}}
\nc{\Span}{\mathbf{Corr}}
\nc{\Gap}{\mrm{Gap}}
\nc{\Corrfr}{\Corr^{\fr}}
\nc{\Corrvfr}{\Corr^{\Vfr}}
\nc{\Spec}{\on{Spec}}
\nc{\Sm}{\on{Sm}}
\nc{\Gm}{\mathbf{G}_{\on{m}}}
\renc{\P}{\bP}
\nc{\nis}{\mathrm{nis}}
\nc{\zar}{\mathrm{zar}}
\nc{\et}{\mathrm{\acute et}}
\nc{\all}{\mathrm{all}}
\nc{\fold}{\mathrm{fold}}
\nc{\Fun}{\mathrm{Fun}}
\nc{\Ho}{\mathrm{Ho}}
\nc{\Segal}{\mathrm{Segal}}
\nc{\Mon}{\mrm{Mon}{}}
\nc{\Ab}{\mrm{Ab}}
\nc{\Sh}{\on{Sh}}
\nc{\M}{\mrm{M}}
\nc{\Lhtp}{L_{\A^1}}
\nc{\Lmot}{L_{\mrm{mot}}}
\nc{\mot}{\mrm{mot}}
\nc{\SH}{\mbf{SH}}
\nc{\RR}{\mbf{R}}
\nc{\CC}{\mbf{C}}
\nc{\Mod}{\mbf{Mod}}
\nc{\QCoh}{\mbf{QCoh}}
\nc{\MonUnit}{\mbf{1}}
\nc{\1}{\mbf{1}}
\nc{\tr}{\on{tr}}
\nc{\cotr}{\mrm{cotr}}
\nc{\vop}{\mrm{vop}}
\nc{\fr}{{\on{fr}}}
\nc{\Ar}{\mrm{Ar}}
\nc{\Vfr}{\on{Vfr}}
\nc{\frdiff}{{\on{frdiff}}}
\nc{\frGys}{\on{frGys}}
\nc{\SHfr}{\SH^{\fr}}
\nc{\SHfrdiff}{\SH^{\frdiff}}
\nc{\SHfrGys}{\SH^{\frGys}}
\nc{\InftyCat}{(\mathrm{\infty,1)\textnormal{-}Cat}}
\nc{\TriCat}{\mathrm{TriCat}}
\nc{\Cat}{\mathrm{1\textnormal{-}Cat}}
\nc{\Th}{\on{Th}}

\nc{\CMon}{\mrm{CMon}{}}
\nc{\MGL}{\mrm{MGL}}
\nc{\Seg}{\mrm{Seg}{}}
\nc{\GW}{\mrm{GW}{}}
\nc{\Tw}{\mrm{Tw}}
\nc{\sslash}{/\mkern-6mu/}
\nc{\PrL}{\mrm{Pr}^\mrm{L}}
\nc{\PrR}{\mrm{Pr}^\mrm{R}}
\nc{\pr}{\mrm{pr}}
\let\phi\varphi

\nc\efr{\mrm{efr}}
\nc\nfr{\mrm{nfr}}
\nc\dfr{\mrm{fr}}
\nc\tfr{\mrm{tfr}}
\nc\Vect{\mrm{Vect}}
\nc\sVect{\mrm{sVect}}
\nc{\fix}{\mrm{fix}}
\nc{\ho}{\mrm{h}}
\nc\Mfd{\mrm{Mfd}}
\nc{\PSh}{\mrm{PSh}}
\nc{\hzmw}{H \tilde\Z{}}
\nc{\Cor}{\mrm{Cor}{}}
\nc{\cormw}{\mrm{\widetilde{Cor}}{}}
\nc{\Chw}{\mrm{\widetilde{CH}}{}}
\nc{\Ex}{\mrm{Ex}}
\nc{\BM}{\mrm{BM}}
\let\setminus-
\nc{\Pic}{\mrm{Pic}}
\nc{\pur}{\mathfrak p}
\nc{\angles}[1]{\langle #1\rangle}
\nc{\inv}[1]{[\tfrac{1}{#1}]}
\nc{\pinv}{\inv{p}}
\nc{\cinv}{\inv{p}}
\nc{\Sph}{\on{Sph}}
\nc{\cdh}{\mrm{cdh}}
\nc{\KGL}{\mrm{KGL}}
\nc{\KH}{\mrm{KH}}
\nc{\Flag}{\mrm{Flag}}

\nc{\inftyCat}{\term{$\infty$-category}}
\nc{\inftyCats}{\term{$\infty$-categories}}

\nc{\inftyOneCat}{\term{$(\infty,1)$-category}}
\nc{\inftyOneCats}{\term{$(\infty,1)$-categories}}

\nc{\inftyGrpd}{\term{$\infty$-groupoid}}
\nc{\inftyGrpds}{\term{$\infty$-groupoids}}

\nc{\inftyTop}{\term{$\infty$-topos}}
\nc{\inftyTops}{\term{$\infty$-toposes}}

\nc{\inftyTwoCat}{\term{$(\infty,2)$-category}}
\nc{\inftyTwoCats}{\term{$(\infty,2)$-categories}}

\title{Perfection in motivic~homotopy~theory}

\subjclass[2010]{Primary 14F42; Secondary 19E08}

\author[E. Elmanto]{Elden Elmanto}
\address{K\o benhavns Universitet\\
Institut for Matematiske Fag\\
Universitetsparken 5\
2100 K\o benhavn\\
Denmark}
\email{\href{mailto:elmanto@math.ku.dk}{elmanto@math.ku.dk}}
\urladdr{\url{https://www.eldenelmanto.com/}}

\author[A. A. Khan]{Adeel A. Khan}
\address{Fakultät für Mathematik\\
Universität Regensburg\\
Universitätsstr. 31\\
93040 Regensburg\\
Germany}
\email{\href{mailto:adeel.khan@mathematik.uni-regensburg.de}{adeel.khan@mathematik.uni-regensburg.de}}
\urladdr{\url{https://www.preschema.com}}

\date{\today}

\begin{document}

\begin{abstract}
We prove a topological invariance statement for the Morel--Voevodsky motivic homotopy category, up to inverting exponential characteristics of residue fields.
This implies in particular that $\SH\pinv$ of characteristic $p>0$ schemes is invariant under passing to perfections.
Among other applications we prove Grothendieck--Verdier duality in this context.
\end{abstract}

\maketitle

\parskip 0pt
\tableofcontents

\parskip 0.2cm

\changelocaltocdepth{1}

\section{Introduction}

Let $S$ be a scheme and denote by $S_\et$ its small étale topos.
The starting point for this note is Grothendieck's ``\'equivalence remarquable de cat\'egories'' \cite[Th\'eor\`eme 18.1.2]{EGAIV4}, which asserts that for any nil-immersion $f : S_0 \hook S$, there is an induced equivalence
  \begin{equation*}
    f^* : S_\et \to (S_0)_\et.
  \end{equation*}
In fact, Grothendieck further generalized this to a \emph{topological invariance} statement for the small étale topos: for any universal homeomorphism of schemes $f : T \to S$, the functor $f^* : S_\et \to T_\et$ is an equivalence (see \cite[Exposé~IX,~Théorème~4.10]{SGA1}, \cite[Exposé~VIII,~Théorème~1.1]{SGA4}).

The large étale topos fails to satisfy nil-invariance.
An observation of Morel and Voevodsky \cite{MV} was that this failure can be repaired by working in the setting of \emph{$\A^1$-invariant} sheaves.
Indeed, it is a consequence of the Morel--Voevodsky localization theorem that the stable motivic homotopy category $\SH$ satisfies nil-invariance (see e.g. \cite[Proposition~2.3.6(1)]{CD}).
However, the topological invariance property still fails, at least in positive characteristic (see \remref{rem:warn}).
Our goal in this paper is to show that topological invariance is in fact true for $\SH$, after inverting the exponential characteristic of the base field (\thmref{thm:radicial invariance}).
This also recovers the analogous statement in other related contexts, such as various variants of mixed motives \cite{AyoubEtale,CD,CDintegral,CDetale} (see \remref{rem:motivic category}).

A particularly useful consequence of topological invariance is that, for any scheme $S$ of characteristic $p$, $\SH(S)\pinv$ is invariant under passing to the perfection $S_\perf$ (\corref{cor:perfection invariance}).
This allows us to remove perfectness hypotheses on the base field in many results, see \ssecref{ssec:non-perfect}.
It also yields a Grothendieck--Verdier duality statement, following ideas of Cisinski--Déglise \cite{CDintegral} (\thmref{thm:duality}).

\ssec{Conventions}

\sssec{}
All schemes will implicitly be assumed to be quasi-compact quasi-separated.

Recall that a morphism of schemes $f : X \to Y$ is a \emph{universal homeomorphism} if it induces a homeomorphism on underlying topological spaces after any base change, or equivalently, if it is integral, universally injective, and surjective \cite[Corollary~18.12.11]{EGAIV4}.

\sssec{}
If $S$ is a scheme of characteristic $p>0$, i.e., an $\bF_p$-scheme, we write $F_S : S \to S$ for the Frobenius endomorphism \cite[Exposé~XIV=XV]{SGA5}.
Recall that $S$ is \emph{perfect} if the Frobenius $F_S:S \rightarrow S$ is an isomorphism.
Any $\bF_p$-scheme $S$ admits a \emph{perfection} $S_\perf$, defined as the limit of the tower
  \begin{equation*}
    \cdots \xrightarrow{F_S} S \xrightarrow{F_S} S,
  \end{equation*}
see \cite[Section~3]{bhatt-scholze-proj}.

\sssec{}
Given a scheme $S$, we denote by $\SH(S)$ the stable \inftyCat of motivic spectra over $S$.
We will use the language of six operations, see \cite[Appendix~C]{HoyoisGLV} or \cite{KhanThesis} for the non-noetherian setting.
Any motivic spectrum $E \in \SH(S)$ represents a cohomology theory on $S$-schemes, given by the formula
  \begin{equation*}
    E(X, \xi) = \Maps_{\SH(S)}(\MonUnit_S, f_*\Sigma^\xi f^*(E))
  \end{equation*}
for any morphism $f : X \to S$ and any K-theory class $\xi \in \K(X)$.
Similarly, there is a Borel--Moore homology theory
  \begin{equation*}
    E^\BM(X/S, \xi) = \Maps_{\SH(S)}(\MonUnit_S, f_*\Sigma^{-\xi} f^!(E)).
  \end{equation*}
We refer to \cite{DJK} for details.

\ssec{Acknowledgements}
We would like to thank Denis-Charles Cisinski, Frédéric Déglise, Marc Hoyois, and Fangzhou Jin for useful suggestions and conversations on the subject of this paper. We would also like to thank the Center for Symmetry and Deformation at the University of Copenhagen for supporting AK's visit during which this paper was conceived. EE would like to thank a couple of cheeky colleagues for suggesting the title.

\changelocaltocdepth{2}


\section{Topological invariance}
\label{sec:perfinv}

\ssec{Main result and corollaries}

Let $\sP$ be a set of prime numbers.
For a scheme $S$, we denote by $\SH(S)[\sP^{-1}]$ the localization of $\SH(S)$ at the morphisms $p : E \to E$, for $E \in \SH(S)$ and $p \in \sP$.
When $\sP$ contains a single prime $p$, we write simply $\SH(S)\pinv$.

\begin{thm}\label{thm:radicial invariance}
Let $S$ be a scheme and $\sP$ a set of prime numbers.
Suppose that every prime $q \notin \sP$ is invertible in $\sO_S$.
Then for any universal homeomorphism $f : T \to S$, the functor
  \[ f^*: \SH(S)[\sP^{-1}] \rightarrow \SH(T)[\sP^{-1}] \]
is an equivalence.
\end{thm}

\begin{rem}\label{rem:motivic category}
The proof of \thmref{thm:radicial invariance} will in fact apply to any motivic \inftyCat of coefficients $\bD$ as in \cite[Chap.~2,~Definition~3.5.2]{KhanThesis}.
See \remref{rem:motivic category proof} for details.
\end{rem}

We now record some immediate consequences.
Taking $\sP$ to be the set of all primes, we get:

\begin{cor} \label{cor:q}
For any universal homeomorphism $f : T \to S$, the functor 
\[
f^*: \SH(S)_{\bQ} \rightarrow \SH(T)_{\bQ}
\]
is an equivalence.
\end{cor} 

\begin{rem}
Recall that for every scheme $S$, the category $\SH(S)_\bQ$ admits natural splittings
  \begin{equation*}
    \SH(S)_\bQ \simeq \SH(S)_{\bQ,+} \times \SH(S)_{\bQ,-},
  \end{equation*}
see \cite[Sect.~16.2]{CD}.
The analogue of \corref{cor:q} is known for the plus part $\SH(S)_{\bQ,+}$, via the identification with Beilinson motives (see Theorems 14.3.3 and 16.2.13 in \emph{op. cit.}).
For the minus part, the statement appears to be new.
\end{rem}

Taking $\sP$ to be a single prime, we get:

\begin{cor}\label{cor:univ homeo invariance}
Let $S$ be a scheme of exponential characteristic $p$.
Then for any universal homeomorphism $f : T \to S$, the functor
  \[ f^*: \SH(S)\pinv \rightarrow \SH(T)\pinv \]
is an equivalence.
\end{cor}

\begin{cor}\label{cor:Frob invariance}
For every scheme $S$ of characteristic $p>0$, the absolute Frobenius induces an equivalence
  \begin{equation*}
    F_S^* : \SH(S)\pinv \to \SH(S)\pinv.
  \end{equation*}
\end{cor}

\begin{cor} \label{cor:perfection invariance}
For every scheme $S$ of characteristic $p>0$, the canonical morphism $S_\perf \to S$ induces an equivalence
  \begin{equation*}
    \SH(S)\pinv \rightarrow \SH(S_{\perf})\pinv.
  \end{equation*}
\end{cor}

\begin{proof}
Follows from \corref{cor:Frob invariance} in view of continuity of $\SH$ \cite[Proposition~C.12(4)]{HoyoisGLV}.
\end{proof}

\sssec{}

At the level of cohomology and Borel--Moore homology, we have the following reformulation (we consider the case $\sP=\{p\}$ for simplicity):

\begin{cor}\label{cor:cohomology}
Let $S$ be a scheme of exponential characteristic $p$.
Let $E \in \SH(S)$ be a motivic spectrum over $S$.
Then we have:
\begin{thmlist}
  \item\label{item:cohomology/radicial}
For any universal homeomorphism $f : X \to Y$ of $S$-schemes, the induced maps
  \begin{align*}
    f^* : E(Y, \xi)\pinv \to E(X, f^*(\xi))\pinv\\
    f_* : E^{\BM}(X, f^*(\xi))\pinv \to E^{\BM}(Y, \xi)\pinv
  \end{align*}
are equivalences for every $\xi \in \K(Y)$.
  \item\label{item:cohomology/perfection}
The canonical morphism $f : S_\perf \to S$ induces equivalences
  \begin{align*}
    f^* : E(S, \xi)\pinv \to E(S_\perf, f^*(\xi))\pinv\\
    f_* : E^{\BM}(S_\perf, f^*(\xi))\pinv \to E^{\BM}(S, \xi)\pinv
  \end{align*}
for every $\xi \in \K(S)$.
\end{thmlist}
\end{cor}

\begin{proof}
Note that we have canonical identifications
  \begin{equation*}
    E(T, \psi)\pinv
      \simeq \Maps_{\SH(S)}(\1_S, (\pi_T)_*\Sigma^\psi \pi_T^*(E))\pinv
      \simeq \Maps_{\SH(S)\pinv}(\1_S, (\pi_T)_*\Sigma^\psi \pi_T^*(E))
  \end{equation*}
for every $\pi_T : T \to S$ and $\psi\in\K(T)$.
Therefore the map on cohomology spaces is induced from the natural transformation
  \begin{equation*}
      (\pi_Y)_*\Sigma^\xi \pi_Y^*
        \to (\pi_Y)_*\Sigma^\xi f_*f^* \pi_Y^*
        \simeq (\pi_Y)_* f_*\Sigma^{f^*(\xi)} f^* \pi_Y^*
        \simeq (\pi_X)_* \Sigma^{f^*(\xi)} \pi_X^*.
    \end{equation*}
For $f$ as in \itemref{item:cohomology/radicial} (resp. \itemref{item:cohomology/perfection}), the unit map $\id \to f_*f^*$ is invertible after inverting $p$ by \thmref{thm:radicial invariance} (resp. by \corref{cor:perfection invariance}), whence the claim.
The proof for Borel--Moore homology is similar, using the fact that the co-unit map $f_*f^! \to \id$ is also invertible in both cases (after inverting $p$).
\end{proof}

\begin{rem}
\corref{cor:cohomology} also holds for the compactly supported variants (cohomology with compact support and relative homology), with the same proofs.
\end{rem}

\begin{exam}\label{exam:perfection invariance in K}
Let $\KGL \in \SH(\bF_p)$ denote the homotopy invariant K-theory spectrum over $\Spec(\bF_p)$.
For every (possibly singular and non-noetherian) $\bF_p$-scheme $S$, we have functorial equivalences
  \begin{equation*}
    \KGL(S, 0)\pinv \simeq \K(S)\pinv
  \end{equation*}
by \cite[Theorem~2.20]{Cisinski} and \cite[Exercise~9.11(h)]{TT}.
Under these identifications, \corref{cor:cohomology} recovers in particular the recent observation of Kelly and Morrow \cite[Lemma~4.1]{kelly-morrow} that the canonical map
  \begin{equation*}
    \K(S)\pinv \to \K(S_\perf)\pinv
  \end{equation*}
is an equivalence.
\end{exam}

\begin{rem} \label{rem:warn}
\corref{cor:univ homeo invariance} is \emph{false} before inverting $p$.
Indeed, let $k$ be a field such that the Frobenius induces an equivalence $F^* : \SH(k) \to \SH(k)$.
Then as in \corref{cor:cohomology}, the induced map on algebraic K-theory spectra
  \begin{equation*}
    F^* : \K(k) \to \K(k)
  \end{equation*}
is also an equivalence.
But under the canonical identification $\K_1(k) \simeq k^\times$, the induced endomorphism of $k^\times$ is $x \mapsto x^p$, which is an isomorphism if and only if $k$ is perfect.
\end{rem}

\begin{rem}
Under the assumptions of \thmref{thm:radicial invariance}, suppose further that $f$ is of finite type (hence a finite radicial surjection).
Then the equivalence $f^* : \SH(S)[\sP^{-1}] \to \SH(T)[\sP^{-1}]$ is quasi-inverse to $f_* \simeq f_!$.
Hence the left adjoint and right adjoints of the latter functor are equivalent.
That is, we have an equivalence of functors 
\[
f^* \simeq f^!:  \SH(S)[\sP^{-1}] \rightarrow \SH(T)[\sP^{-1}].
\]
\end{rem}

\ssec{Proof of \thmref{thm:radicial invariance}}

In order to prove \thmref{thm:radicial invariance}, we have to show that the unit and co-unit maps
  \begin{equation*}
    \id \to f_*f^*,
    \qquad
    f^*f_* \to \id
  \end{equation*}
are both invertible.
For the latter, this turns out to be the case before inverting $p$:

\begin{prop} \label{prop:co-unit invertible}
For any universal homeomorphism $f : T \to S$, the co-unit transformation
  \[ f^*f_* \to \id_{\SH(T)} \]
is invertible.
In other words, the functor $f_* : \SH(T) \to \SH(S)$ is fully faithful.
\end{prop}

\begin{proof} 
We argue as in the proof of \cite[Proposition~2.1.9]{CD}.
By the proper base change formula, this co-unit is identified with the natural transformation $(\pi_2)_*(\pi_1)^* \to \id$, where $\pi_1$ and $\pi_2$ are the respective projections $T \fibprod_S T \to T$.
Since $f$ is a universal homeomorphism, its diagonal $\Delta : T \to T\fibprod_S T$ is a nilpotent closed immersion.
Then by the localization theorem (cf. \cite[Proposition~2.3.6(1)]{CD}), $\Delta^*$ is an equivalence.
Since $\pi_1$ and $\pi_2$ are retractions of $\Delta$ it follows that we have canonical identifications $\Delta^* \simeq (\pi_\varepsilon)_*$ and $\Delta_* \simeq (\pi_\varepsilon)^*$ for each $\varepsilon\in\{1,2\}$.
In particular, the natural transformation $(\pi_2)_*(\pi_1)^* \to \id$ is identified with the co-unit $\Delta^*\Delta_* \to \id$, which is invertible.
\end{proof}

For the unit map, we will require a more involved argument.
We begin by introducing some notation.

\begin{notat}\label{notat:n_epsilon}
Given a unit $a \in \Gamma(S, \sO_S)^\times$, we write $\angles{a}$ for the induced point of $\Omega\K(S)$.
For an integer $n\ge 0$, we write $n_\epsilon$ for the formal sum
  \begin{equation*}
    n_\epsilon = 1 + \angles{-1} + 1 + \cdots
  \end{equation*}
which consists of $n$ terms.
\end{notat}

\begin{rem}
Recall that there is a canonical map of spaces $\K(S) \to \Aut(\SH(S))$ which sends the class of a perfect complex $\xi$ to the auto-equivalence $\Sigma^{\xi} = \Th_S(\xi) \otimes -$ of $\SH(S)$.
It sends $0$ to the identity $\id_{\SH(S)}$ and thus induces a canonical map
  \begin{equation}
    \Omega\K(S) \to \Aut(\id_{\SH(S)}),
  \end{equation}
via which we may also regard $n_\epsilon$ (\notatref{notat:n_epsilon}) as an automorphism of $\id_{\SH(S)}$.
\end{rem}

We are grateful to Marc Hoyois for suggesting the following re-interpretation of \cite[Proposition~B.1.4]{EHKSY1}.

\begin{prop}\label{prop:trace transformation}
Let $S$ be a scheme, $P \in \Gamma(S,\sO_S)[x]$ a monic polynomial of degree $d$, and $T \subset \A^1_S$ the closed subscheme cut out by $P$.
If $f : T \to S$ denotes the canonical morphism, then there exists a canonical natural transformation
  \begin{equation*}
    \tr_f : f_*f^* \to \id_{\SH(S)}
  \end{equation*}
such that the composites
  \begin{equation*}
    \id \xrightarrow{\mrm{unit}} f_*f^* \xrightarrow{\tr_f} \id,
    \qquad
    f_*f^* \xrightarrow{\tr_f} \id \xrightarrow{\mrm{unit}} f_*f^*
  \end{equation*}
are homotopic to $d_{\epsilon}$ and $f_* \ast d_\epsilon \ast f^*$, respectively.
\end{prop}

\begin{proof}
Note that $f : T \to S$ is finite and syntomic.
The conormal sheaf $\sN_{T/\A^1_S} \simeq (P)/(P^2)$ is free of rank $1$, and the generator $P$ induces a canonical trivialization $\tau : \bL_{f} \simeq 0$ in $\K(T)$.
Therefore, the trace transformation $\tr_f$ of \cite[4.3.1]{DJK} induces a canonical natural transformation
  \begin{equation*}
    f_* f^* \simeq f_* \Sigma^{\bL_f} f^* \xrightarrow{\tr_f} \id
  \end{equation*}
which we denote again by $\tr_f$.
The claim can be reformulated in cohomological terms as the assertion that, for every $E \in \SH(S)$, the composites
  \begin{align}
    E(S, 0) \xrightarrow{f^*} E(T, 0)
      \simeq E(T, \angles{\bL_f})
      \xrightarrow{f_!} E(S, 0),\label{eq:composite 1}
    \\
    E(T, 0) \simeq E(T, \angles{\bL_f})
      \xrightarrow{f_!} E(S, 0)
      \xrightarrow{f^*} E(T, 0)\label{eq:composite 2}
  \end{align}
are homotopic to multiplication by $d_\epsilon$.

Regarding the assignment $X \mapsto E(X, 0)$ as a presheaf with framed transfers, the first composite is induced by the framed correspondence
  \begin{equation*}
    \begin{tikzcd}
      & T \ar[swap]{ld}{f,\tau}\ar{rd}{f} &
      \\
      S & & S.
    \end{tikzcd}
  \end{equation*}
Therefore the claim follows from \cite[Proposition~B.1.4]{EHKSY1}.
The second composite is identified, by the transverse base change property of the trace transformation $\tr_f$ \cite[Proposition~2.5.6]{DJK}, with
  \begin{equation*}
    E(T, 0)
      \xrightarrow{\pi_2^*} E(T\fibprod_S T, 0)
      \simeq E(T\fibprod_S T, \angles{L_f})
      \xrightarrow{(\pi_1)_!} E(T, 0),
  \end{equation*}
where $\pi_1$ and $\pi_2$ are the first and second projections of $T\fibprod_S T$, respectively.
As above, this is induced by the framed correspondence
  \begin{equation*}
    \begin{tikzcd}
      & T \fibprod_S T \ar{rd}{\pi_2}\ar[swap]{ld}{\pi_1,\pi_2^*(\tau)} &
      \\
      T & & T
    \end{tikzcd}
  \end{equation*}
so the claim follows by another application of \cite[Proposition~B.1.4]{EHKSY1}.
\end{proof}

\begin{lem}\label{lem:q_epsilon invertible}
Let $S$ be the spectrum of a field $k$ of exponential characteristic $p$.
Then for any power $q$ of $p$, the canonical map
  \begin{equation*}
    \End_{\SH(S)}(\MonUnit_S) \to \End_{\SH(S)}(\MonUnit_S)\pinv \simeq \End_{\SH(S)\pinv}(\MonUnit_S)
  \end{equation*}
sends $q_\epsilon \in \End_{\SH(S)}(\MonUnit_S)$ to a unit.
\end{lem}

\begin{proof}
We only need to consider the case $p>1$.
Using Morel's identification $\End(\MonUnit_S) \simeq \GW(k)$ \cite{morel-pi0}, which has been extended in \cite[Lemma~10.12]{norms}, it will suffice to show that the induced element $q_\epsilon \in \GW(k)\pinv$ is invertible.
In view of the cartesian square
 \[
\begin{tikzcd}
\GW(k) \ar{r}{\dim} \ar[twoheadrightarrow]{d} & \bZ \ar[twoheadrightarrow]{d}\\
\W(k) \ar{r} & \bZ/2
\end{tikzcd}
\]
as in \cite[(3.1)]{Morel} (cf. \cite[Lemma~17]{tom-conserve}, \cite[Lemma 1.16]{knebusch}), it will in fact suffice to only check invertibility in $\bZ\pinv$ and in $\W(k)\pinv$.
The former is obvious.
For the latter, we first assume that $p$ is odd.
In this case, we note that $d_{\epsilon} = \frac{d-1}{2}h +1$ and thus $q_{\epsilon}$ is invertible in $\W(k)$ (without inverting $p$).
When $p=2$, $-1$ is trivially a sum of squares in $k$, so the Witt ring is $2$-torsion \cite[Theorem~III.3.6]{milnor1973symmetric} and the claim follows.
\end{proof}

We are now ready to complete the proof of \thmref{thm:radicial invariance}:

\begin{proof}[Proof of \thmref{thm:radicial invariance}]\hypertarget{proof:radicial invariance}
After \propref{prop:co-unit invertible} it remains to show that unit map $\id_{\SH(S)} \rightarrow f_*f^*$ becomes invertible after inverting the primes in $\sP$.
By continuity \cite[Proposition~C.12(4)]{HoyoisGLV} and proper base change, we may use a noetherian approximation argument \cite[Theorem~C.9]{TT} to assume that $S$ is noetherian and of finite dimension.
Then using \cite[Proposition~A.3]{norms} (and proper base change again), we may assume that $S$ is a henselian local scheme (which is still noetherian and finite-dimensional); we denote its closed point by $i : \{s\} \to S$ and the complement by $j : U \to S$.
By the localization theorem, the pair of functors $(i^*, j^*)$ is jointly conservative (see \cite[Section~2.3]{CD}).
Since $U$ has dimension strictly lower than that of $S$, we can argue by induction on the dimension of $S$ to reduce to the case where $S = \{s\}$, i.e., where $S$ is the spectrum of a field $k$.
Since $f$ is radicial, it is then induced by a purely inseparable field extension $k \subset K$.
In characteristic zero ($p=1$), we are already done.
Otherwise, by using continuity again, we may assume that the extension $k \subset K$ is finite, i.e., that $K = k(\alpha)$ with $\alpha^{q} \in k$ for $q$ some power of the prime $p$.
Now it follows from \propref{prop:trace transformation} and \lemref{lem:q_epsilon invertible} that the unit map $\id_{\SH(S)} \to f_*f^*$ is invertible after inverting $p$.
But the assumption implies that $p \in \sP$, so the conclusion follows.
\end{proof}

\begin{rem}\label{rem:motivic category proof}
We now explain how the above proof can be generalized to an arbitrary motivic \inftyCat of coefficients $\bD$ as in \remref{rem:motivic category}.
First, we recall that the theory of fundamental classes developed in \cite{DJK} applies in this more general setting (see 4.3.4 in \emph{op. cit.}).
Therefore, for any object $E \in \bD(S)$, the cohomology theory $X \mapsto E(X, 0)$ defines a presheaf with framed transfers on the category of $S$-schemes.
The proof of \propref{prop:trace transformation} then applies \emph{mutatis mutandis}.
To show that \lemref{lem:q_epsilon invertible} holds for $\bD$, i.e., that $q_\epsilon$ is a unit for every power $q$ of $p$, we argue as follows.
Since $\bD$ satisfies Nisnevich descent \cite[Prop.~2.3.8]{CD}, $\A^1$-invariance, and Thom stability, it follows from the universal property of $\SH$ \cite[Corollary~1.2]{Robalo} that there is a canonical monoidal realization functor $\SH(S) \to \bD(S)$.
Moreover, the canonical map $\Omega\K(S) \to \Aut(\id_{\bD(S)})$ factors through the induced map $\Aut(\id_{\SH(S)}) \to \Aut(\id_{\bD(S)})$, so the claim follows from the universal case of $\SH$.
The rest of the proof of \thmref{thm:radicial invariance} only relies on the formalism of six operations, which is available for $\bD$ \cite[Cor.~4.2.3]{KhanThesis}.
\end{rem}


\section{Applications}
\label{sec:applications}

\ssec{Duality}

Let $S$ be a scheme that is locally of finite type over a field $k$ of exponential characteristic $p$.
The structural morphism $\pi: S \rightarrow \Spec(k)$ determines a \emph{duality functor} defined by
\[
D_S(E) = \uHom(E, \pi^!(\MonUnit_k)).
\]
To justify this name, we must show that the object $\pi^!(\MonUnit_k)$ is \emph{dualizing}.
That is:

\begin{thm} \label{thm:duality}
For any compact object $E \in \SH(S)$, the canonical map
\[
E \rightarrow D_S(D_S(E))
\]
is an equivalence in $\SH(S)\pinv$.
\end{thm}

\begin{rem}
As remarked in \cite[Remark~7.4]{CDintegral}, Theorem~\ref{thm:duality} implies the formalism of Grothendieck--Verdier duality for $\SH\pinv$, for locally of finite type $k$-schemes.
In particular, this gives an improvement of \cite[Theorem~2.4.8]{deglise2015dimensional}.
\end{rem}

\thmref{thm:duality} follows from the following statement, analogous to \cite[Proposition~7.2]{CDintegral}.

\begin{prop} \label{prop:generate}
The full subcategory of compact objects in $\SH(S)\pinv$ is generated as a thick subcategory by objects of the form $f_!(\MonUnit)(n)$, where $f: X \rightarrow S$ proper, $X$ is smooth over a purely inseparable extension of $k$, and $n \in \bZ$ is an integer.
\end{prop}

\begin{proof}
If $k$ is perfect, the statement is \cite[Corollary~2.4.7]{deglise2015dimensional}.
In general, the morphism $\varphi : S_\perf \simeq S \fibprod_{\Spec(k)}\Spec(k_\perf) \to S$ induces an equivalence $\varphi^* : \SH(S) \to \SH(S_\perf)$ by \corref{cor:perfection invariance}.
If $f : X \to S_\perf$ is a proper morphism with $X$ smooth over $k_\perf$, then the composite $X \to S_\perf \to S$ is as in the statement, so we conclude.
\end{proof}

\begin{proof}[Proof of \thmref{thm:duality}]
As in the proof of \cite[Theorem~7.3]{CDintegral}, this follows immediately from Proposition~\ref{prop:generate}, and Ayoub's purity theorem for smooth morphisms \cite[Section 1.6]{Ayoub}, \cite[Appendix C]{HoyoisGLV}.
\end{proof}

\ssec{Removal of perfectness hypotheses}
\label{ssec:non-perfect}

\corref{cor:perfection invariance} allows us to immediately drop perfectness hypotheses in many known results, at least after inverting the exponential characteristic.
Some examples are listed below.

\begin{thm} \label{thm:dualizability}
Let $S$ be the spectrum of a field $k$ of exponential characteristic $p$.
For any smooth $S$-scheme $X$, the suspension spectrum $\Sigma^{\infty}_+(X)$ is strongly dualizable in $\SH(S)\pinv$.
In particular, $\SH(S)\pinv$ is generated under colimits by the strongly dualizable objects.
\end{thm}

Indeed, we can use \corref{cor:perfection invariance} to reduce the case where $k$ is perfect, which is due to Riou, see \cite[Corollary~B.2]{levine2013algebraic}.

\begin{rem}
\propref{prop:trace transformation} gives the following refinement of \cite[Lemma~B.3]{levine2013algebraic}.
Suppose $f : Y \to X$ is a finite étale morphism of degree $d$ between smooth connected $k$-schemes.
Then, up to replacing $X$ by a dense open subset $U \subseteq X$ and $f$ by its base change $f_U : Y_U \to U$, there are isomorphisms of natural transformations
  \begin{align*}
    \big(\id
      \xrightarrow{\mrm{unit}} f_*f^*
      \simeq f_\sharp f^*
      \xrightarrow{\mrm{counit}} \id\big)
    &\simeq 
    d_\epsilon,\\
    \big(f_*f^*
      \simeq f_\sharp f^*
      \xrightarrow{\mrm{counit}} \id
      \xrightarrow{\mrm{unit}} f_*f^* \id\big)
    &\simeq 
    f_* \ast d_\epsilon \ast f^*,
  \end{align*}
where $d_\epsilon$ is as in \notatref{notat:n_epsilon}.
To prove this, note there are canonical identifications $\Sigma^{L_f} \simeq \id$ and $f_* \simeq f_! \simeq f_\sharp$ since $f$ is finite and étale, and the composite
  \begin{equation*}
    f_*f^* \simeq f_\sharp f^* \xrightarrow{\mrm{counit}} \id
  \end{equation*}
is canonically homotopic to the trace transformation $\tr_f : f_*f^* \simeq f_*\Sigma^{L_f}f^* \to \id$.
Replacing $X$ by its generic point, we may assume that $X = \Spec(k)$.
Then $Y = \Spec(K)$ with $K/k$ a finite separable field extension, so by the primitive element theorem we are now in the situation of \propref{prop:trace transformation}.
\end{rem}

\sssec{}

We also have the following variant of Bachmann's conservativity theorem \cite{tom-conserve}.

\begin{thm} \label{thm:perf}
Let $k$ be a field with finite $2$-\'etale cohomological dimension and exponential characteristic $p$. Then the canonical functor
\[
\SH(k)\pinv \rightarrow \DM(k; \bZ\pinv)
\]
is conservative on compact objects.
\end{thm}

\begin{proof} Using Corollary~\ref{cor:perfection invariance} and the analogous result for mixed motives \cite[Lemma~3.15]{CDintegral}, we may replace $k$ by its perfection.
Then the result is proven in \cite[Theorem 16]{tom-conserve}.
\end{proof}

\begin{rem} \label{rem:tom} Using Theorem~\ref{thm:perf} we can deduce the Pic-injectivity result of \cite[Theorem 18]{tom-conserve}. This extends Bachmann's results on Po Hu's conjecture on invertibility of the the suspension spectra of affine quadrics to imperfect fields; see \emph{loc. cit} for details.
\end{rem}

\sssec{}

We can similarly extend the recognition principle for infinite loop spaces \cite[Theorem~3.5.13]{EHKSY1} to non-perfect fields.

\begin{thm} \label{thm:recog-perf}
Let $k$ be a field of exponential characteristic $p$.
Then there are canonical equivalences of symmetric monoidal $\infty$-categories
  \begin{align*}
    \H^{\fr}(k)^{\gp}\cinv &\simeq \SH^{\veff}(k)\cinv,\\
    \SH^{S^1,\fr}(k) \cinv &\simeq \SH^{\eff}(k)\cinv.
  \end{align*}
\end{thm}

\begin{rem} \label{rem:cancel}
Theorem~\ref{thm:recog-perf} also implies cancellation in the sense of \cite[Theorem~3.5.8]{EHKSY1} for non-perfect fields, after inverting the exponential characteristic.
\end{rem}

We thank Marc Hoyois for pointing out a gap in the original proof of the following lemma.

\begin{lem}\label{lem:Hfr invariance}
Let $k$ be a field of exponential characteristic $p$.
Then the morphism $f : \Spec(k_\perf) \to \Spec(k)$ induces an equivalence
  \begin{equation*}
    f^* : \H^{\fr}(k)^{\gp}\cinv \rightarrow \H^{\fr}(k_{\perf})^{\gp}\cinv
  \end{equation*}
of symmetric monoidal \inftyCats.
\end{lem}

\begin{proof}
We first show fully faithfulness.
By continuity, it suffices to show that the Frobenius induces fully faithful functors $F^*: \H^{\fr}(k)^{\gp}\cinv \rightarrow \H^{\fr}(k)^{\gp}\cinv$, i.e., that the counit map $\id \rightarrow F_{*}F^*$ is an equivalence after inverting $p$.
For this it suffices to show that, for every grouplike framed motivic space $\sF$ and every smooth $k$-scheme $X$, the induced map of spaces
  \begin{equation*}
    \Gamma(X, \sF) \to \Gamma(F^{-1}(X), F^*(\sF))
  \end{equation*}
is an equivalence after inverting $p$.
Let $\sF^+$ denote the motivic localization of the left Kan extension of $\sF$ along the inclusion from framed correspondences of smooth $k$-schemes to framed correspondences of all $k$-schemes of finite type.
Since the functor $\sF \mapsto \sF^+$ is fully faithful and commutes with $F^*$ (compare \cite[Prop.~11.1.19]{CD}), the map above is identified with the canonical map
  \begin{equation*}
    F^* : \Gamma(X, \sF^+) \to \Gamma(F^{-1}(X), F^*(\sF^+)) \simeq \Gamma(F^{-1}(X), \sF^+).
  \end{equation*}
Arguing just as in the proof of \propref{prop:trace transformation}, we see that the structure of framed transfers on $\sF^+$ gives rise to the two composites \eqref{eq:composite 1} and \eqref{eq:composite 2}, so that \cite[Proposition~B.1.4]{EHKSY1} shows that the map in question is an equivalence after inverting $p_\epsilon$.
We conclude by using the analogue of Lemma~\ref{lem:q_epsilon invertible} for $\End_{\H^{\fr}(k)^{\gp}}(\MonUnit_k)$, which holds because there is a canonical equivalence
\[
\End_{\H^{\fr}(k)^{\gp}}(\MonUnit_k) \rightarrow \End_{\SH(k)}(\MonUnit_k)
\]
by \cite[Theorem~3.5.17]{EHKSY1}.

It remains now to show that $f^*$ is essentially surjective.
Since any smooth irreducible $k_\perf$-scheme is, up to a universal homeomorphism, the base change of a smooth irreducible $k$-scheme \cite[Lemma~1.12]{suslin-last}, it will suffice to show the following claim: for any universal homeomorphism of smooth schemes over $k_\perf$, the induced map in $\H^\fr(k_\perf)^\gp\pinv$ is invertible.
By \cite[Theorem~3.5.13(i)]{EHKSY1} it suffices to show that the induced map in $\SH(k_\perf)\pinv$ is invertible.
This follows directly from \thmref{thm:radicial invariance}.
\end{proof}

\begin{proof}[Proof of \thmref{thm:recog-perf}]
Note that we need only prove the claim when $p > 1$, and that the second claim follows from the first by stabilization.
The equivalence of Corollary~\ref{cor:perfection invariance} restricts to an equivalence
  \begin{equation*}
    \SH^\veff(k)\pinv \rightarrow \SH^\veff(k_{\perf})\pinv
  \end{equation*}
by construction.
Combining this with \lemref{lem:Hfr invariance}, we see that we may replace $k$ by its perfection.
In that case, the statement is \cite[Theorem~3.5.13(i)]{EHKSY1}.
\end{proof}


\bibliographystyle{alphamod}

\let\mathbb=\mathbf

{\small
\bibliography{references}
}

\parskip 0pt

\end{document}